\documentclass{article}

\usepackage{arxiv}

\usepackage[utf8]{inputenc} 
\usepackage[T1]{fontenc}    
\usepackage{hyperref}       
\usepackage{url}            
\usepackage{booktabs}       
\usepackage{amsfonts}       
\usepackage{nicefrac}       
\usepackage{microtype}      
\usepackage{amsmath}
\usepackage{amssymb}
\usepackage{amsthm}
\usepackage{comment}
\usepackage{xparse}

\NewDocumentCommand{\paren}{s m}{%
  \IfBooleanTF{#1}{%
    \mathopen{} (#2)\mathclose{}%
  }%
  {%
    \mathopen{}\left(#2 \right)\mathclose{}%
  }%
}

\NewDocumentCommand{\apply}{s m m}{%
  \IfBooleanTF{#1}{%
    \,#2\paren*{#3}\,%
  }%
  {%
    \,#2\paren{#3}\,
  }%
}

\NewDocumentCommand{\bracketparen}{s m}{%
  \IfBooleanTF{#1}{%
    \mathopen{} [#2]\mathclose{}%
  }%
  {%
    \mathopen{}\left[#2 \right]\mathclose{}%
  }%
}

\NewDocumentCommand{\bracketapply}{s m m}{%
  \IfBooleanTF{#1}{%
    \,#2\bracketparen*{#3}\,%
  }%
  {%
    \,#2\bracketparen{#3}\,
  }%
}

\NewDocumentCommand{\mybrace}{s m}{%
  \IfBooleanTF{#1}{%
    \mathopen{} \{#2\}\mathclose{}%
  }%
  {%
    \mathopen{}\left\lbrace{#2} \right\rbrace\mathclose{}%
  }%
}

\NewDocumentCommand{\funch}{m}{\apply{h}{#1}}
\NewDocumentCommand{\funcphi}{m}{\apply{\varphi}{#1}}
\NewDocumentCommand{\funcexp}{m}{\apply{\exp}{#1}}
\NewDocumentCommand{\funclog}{m}{\apply{\log}{#1}}

\theoremstyle{plain}
\newtheorem{theorem}{Theorem}
\newtheorem{lemma}{Lemma}

\theoremstyle{definition}

\theoremstyle{remark}
\newtheorem*{note}{Note}

\newcommand*{\Real}{\mathbb R}
\newcommand*{\xizero}{{\xi_0}}
\newcommand*{\Nat}{\mathbb N}
\newcommand*{\NatInf}{\Nat_\infty}

\newcommand{\theFraction}{\frac{1}{2m}}

\title{Laplace's formula : an approach by nonstandard analysis}

\author{
  OZAKI Ryushi \\
  \texttt{ozaki.ryushi@gmail.com}
}

\AtBeginDocument{\pagenumbering{arabic}}
\begin{document}
\maketitle

\begin{abstract}

We give a proof of Laplace's formula using nonstandard analysis (NSA).
The use of NSA makes it possible to keep the classical heuristic argument essentially unchanged while avoiding the technical management of limiting processes required in a standard proof.
Only elementary tools from NSA are used.

\end{abstract}


\section{Introduction}

Laplace's formula gives the asymptotic behavior of the integral
\begin{align}
&\int_a^b \funcphi{x}  \funcexp{n\,\funch{x}}  dx \label{ex:001}
\end{align}
as $n$ tends to infinity, under the assumption that $h$ attains its maximum at a unique point in the interval of integration.

For example, if $a=0,\  b=\infty,\  \funcphi{x}\equiv 1$, and $\funch{x}=\funclog{x} - x$, then the integral becomes
\begin{align}
&\int_0^\infty x^n \, e^{-x} \: dx = \apply{\Gamma}{n+1},
\end{align}
and one obtains an alternative derivation of Stirling's formula.
Laplace's formula has many applications in physics and engineering; several examples are discussed in \cite{Polya-and-Szego-1}.

We begin with a preliminary discussion.
The precise formulation and proof are given in Section \ref{sec:statement-and-the-proof}.

Suppose that $\funch{x}$ attains its maximum at a unique point $\xizero \in (a, b)$ and is sufficiently smooth at that point.
Consider
\begin{align}
&\funcexp{-n\,\funch{\xizero}} \int_a^b \funcphi{x}  \funcexp{n\, \funch{x}}  dx = \int_a^b \funcphi{x} \funcexp{n \, \left( h(x) - h(\xizero) \right)}  dx.\label{2025-02-27-004}
\end{align}
For large $n$, the exponential factor on the right-hand side decays rapidly away from $\xizero$.
Thus the main contribution to the integral comes from a small neighborhood of $\xizero$:
\begin{align}
  \text{RHS of \eqref{2025-02-27-004}}&\fallingdotseq \int_{\xizero - \varepsilon}^{\xizero + \varepsilon} \varphi(x) \funcexp{n \, \left(h(x) - h(\xizero) \right)}  dx.\label{2025-02-27-006}
\end{align}
Since $\varepsilon$ is small and $\varphi$ is continuous, $\varphi(x)$ may be replaced by $\varphi(\xizero)$ to first order.
This gives
\begin{align}
\text{RHS of \eqref{2025-02-27-006}} &\fallingdotseq \varphi(\xizero) \int_{\xizero - \varepsilon}^{\xizero + \varepsilon} \funcexp{n \, (h(x) - h(\xizero))}  dx.\label{2025-02-27-001}
\end{align}
Suppose moreover that $h''(\xizero) < 0$.
Then $h(x) - h(\xizero)$ is approximated by
\begin{align}
\frac{h''(\xizero)}{2}(x - \xizero)^2.\label{2025-02-27-002}
\end{align}
Using this approximation, we obtain 
\begin{align}
\text{RHS of \eqref{2025-02-27-001}} &\fallingdotseq \varphi(\xizero) \int_{\xizero - \varepsilon}^{\xizero + \varepsilon} \funcexp{n \, \frac{h''(\xizero)}{2}(x - \xizero)^2} dx.\\
\intertext{By a change of variable,}
&= \varphi(\xizero) \sqrt{\frac{2}{-n\,h''(\xizero)}} \enskip \int_{|z|\le R_{n, \varepsilon}} e^{-z^2}\,dz,\label{2025-02-27-009}
\end{align}
where
\begin{align*}
R_{n, \varepsilon} = \varepsilon \sqrt{\frac{-n\,h''(\xizero)}{2}}.
\end{align*}
If $n$ is taken large enough in comparison with the small parameter $\varepsilon$, then $R_{n, \varepsilon}$ is effectively infinite.
Hence
\begin{align*}
\text{RHS of \eqref{2025-02-27-009}} &\fallingdotseq \varphi(\xizero) \sqrt{\frac{2}{-n \,h''(\xizero)}} \enskip \int_{\Real} e^{-z^2} dz = \varphi(\xizero) \sqrt{ \frac{2\pi}{-n \,h''(\xizero)}}.
\end{align*}

We are thus led to the asymptotic formula
\begin{align*}
\int_a^b \varphi(x) \funcexp{n\,h(x)} dx \sim \varphi(\xizero) \funcexp{n\,h(\xizero)} \sqrt{\frac{2\pi}{-n\,h''(\xizero)}} \qquad (n\to \infty).
\end{align*}

The preceding discussion is heuristic and is not a proof, although it captures the main idea.
Its principal defects are the ambiguous use of $n$ as both an infinitely large and an ordinary large natural number, and the unspecified size of $\varepsilon$.
In a standard proof, these points must be replaced by a careful control of limiting processes, as in \cite{Polya-and-Szego-1}; the resulting argument can be rather difficult to follow.

Nonstandard analysis provides a rigorous way to preserve the intuitive structure of the preceding argument without introducing these technical complications.
Within NSA one may work directly with \emph{infinitely large natural numbers} and \emph{infinitesimal real numbers}; although these notions may appear paradoxical at first, they are well defined in the nonstandard framework.

In this note we prove Laplace's formula by NSA, following the outline just described.
The proof uses only elementary NSA and basic calculus; in particular, no delicate limit argument or Big-O notation is required.

Section \ref{sec:statement-and-the-proof} contains the precise statement and proof of Laplace's formula.
Section \ref{sec:generalization} gives a generalization and its proof.

In what follows we assume basic familiarity with Nelson's Internal Set Theory (IST) \cite{Nelson1977}.
Readers familiar with another standard framework for nonstandard analysis should nevertheless be able to follow the argument without difficulty.

\section{Exact Statement and Proof}\label{sec:statement-and-the-proof}

\begin{theorem}[Laplace's theorem]\label{th:Laplace}
Let $[a,b]\subseteq \Real$ be a finite or infinite interval. Let $\varphi(x)$ and $h(x)$ be functions defined on this interval, and let $\xizero \in (a,b)$. Assume that the following conditions hold:
\begin{enumerate}
 \item[(C1)] $\varphi(x) \funcexp{n\, h(x)}$ is absolutely integrable over $[a,b]$ for $n = 0,1,2,\ldots$.
 \item[(C2)] 
 There exists a neighborhood of $\xizero$ where $h''(x)$ exists and is continuous and $h''(\xizero) < 0$.
\item[(C3)] There exists $r_1 > 0$ such that for any $\rho \in [0, r_1)$,
\begin{align*}
&a \le y \le \xizero - \rho \le x \le \xizero \implies \apply{h}{y} \le \apply{h}{x}.
\end{align*}
\item[(C4)] There exists $r_2 > 0$ such that for any $\rho \in [0, r_2)$,
\begin{align*}
&\xizero \leq x \leq \xizero + \rho \leq y \leq b \implies h(x) \geq h(y).
\end{align*}
\item[(C5)] $\varphi(x)$ is continuous at $x = \xizero$ and $\varphi(\xizero) \neq 0$.
 \end{enumerate}
Then, as $n \to \infty$, the following asymptotic formula holds:
\begin{align*}
\int_a^b \varphi(x) \funcexp{n\, h(x)} dx \:\sim\: \varphi(\xizero) \funcexp{n\, h(\xizero)} \sqrt{-\frac{2\pi}{n h''(\xizero)}}.
\end{align*}
\end{theorem}

\begin{note}
The seemingly complicated conditions (C3) and (C4) may be stated more simply.
For example, (C3) is equivalent to the conjunction of the following two conditions:
\begin{itemize}
  \item [(C3-1)] $\displaystyle \sup_{\enskip a\le \,y\, \le \xizero - r_1} \apply{h}{y} \le \apply{h}{\xizero - r_1}$,
  \item [(C3-2)] $h$ is monotone increasing on the interval $[\xizero - r_1, \xizero]$, in the non-strict sense.
\end{itemize}
\end{note}

\begin{proof}[Proof of Theorem \ref{th:Laplace}]
It is enough to prove
\begin{align}
\lim_{n \to \infty} \sqrt{n} \int_a^b \varphi(x) \funcexp{ n\,(h(x) - h(\xizero)) }  dx = \varphi(\xizero) \sqrt{-\frac{2\pi}{h''(\xizero)}}. \label{ex:purpose}
\end{align}
We first assume that $a$, $b$, $\varphi$, $h$, and $\xizero$ are standard.
Then the desired relation is equivalent to
\begin{align*}
\forall \nu \in \NatInf \quad \left( \sqrt{\nu} \int_a^b \varphi(x) \funcexp{ \nu \, (h(x) - h(\xizero)) }  dx \simeq \varphi(\xizero) \sqrt{-\frac{2\pi}{h''(\xizero)}} \right).
\end{align*}

Fix an arbitrary $\nu\in \NatInf$ and put $\varepsilon := \nu^{-1/6}$.
Then $\varepsilon$ is infinitesimal. By (C2), $h''(x)$ exists and is negative for $|x - \xizero| < \varepsilon$.
By (C3) and (C4), $\xizero$ is a local maximum point of $h$; hence, by (C2), $h'(\xizero)=0$.

We shall use the following estimate on the infinitesimal neighborhood $[\xizero - \varepsilon,\xizero+\varepsilon]$.
\begin{lemma}%
\label{lemma:esimate:001} 
For every $x \in [\xizero - \varepsilon, \xizero + \varepsilon]$, one has
\begin{align}
\frac{3}{2} h''(\xizero) \leq  h''(x) \leq \frac{1}{2} h''(\xizero)\,(<0). \label{estimate:001}
\end{align}
\end{lemma}
\begin{proof}[Proof of Lemma \ref{lemma:esimate:001}]
By the continuity of $h''$ on $[\xizero - \varepsilon, \xizero + \varepsilon]$, the difference $h''(x)-h''(\xizero)$ is infinitesimal. Equivalently,
\begin{align}
&\left\lvert h''(x) - h''(\xizero) \right\rvert \le u,
\end{align}
for every standard positive number $u$.
Since $h''(\xizero)$ is a standard negative number, choosing $u = \lvert h''(\xizero) \rvert/2 = -h''(\xizero)/2$ yields the desired bound on $h''(x)$.
\end{proof}

We now return to the proof of Theorem \ref{th:Laplace}.
We split the integral into three parts: 
\begin{align}
  \sqrt{\nu} \int_a^b \varphi(x) \funcexp{ \nu \, (h(x) - h(\xizero)) }  dx
  &= \sqrt{\nu} \int_a^{\xizero - \varepsilon} \varphi(x) \funcexp{ \nu \, (h(x) - h(\xizero)) }  dx \notag\\
  & + \sqrt{\nu} \int_{\xizero - \varepsilon}^{\xizero + \varepsilon} \varphi(x) \funcexp{ \nu \, (h(x) - h(\xizero)) }  dx \notag\\
  & + \sqrt{\nu} \int_{\xizero + \varepsilon}^b \varphi(x) \funcexp{ \nu \, (h(x) - h(\xizero)) }  dx.\label{eqn:026}
  \end{align}

We first estimate the first term on the right-hand side of \eqref{eqn:026}.
For any $x \in [a, \xizero - \varepsilon]$, we have
\begin{align*}
  h(x) - h(\xizero) &\le h(\xizero - \varepsilon) - h(\xizero) \\
  &= \int_{\xizero}^{\xizero - \varepsilon} \apply{h'}{t} dt = \int_{\xizero}^{\xizero - \varepsilon} \apply{h'}{t} - \underbrace{\apply{h'}{\xizero}}_{=0} dt\\
  &= \int_{\xizero}^{\xizero - \varepsilon} \paren{\int_{\xizero}^t \apply{h''}{u} du  } dt \le  \frac{h''(\xizero)}{2} \int_{\xizero}^{\xizero - \varepsilon} \left( \int_{\xizero}^{t} \, du \right) dt \\
  &= \frac{h''(\xizero)}{4} \varepsilon^2 = \frac{h''(\xizero)}{4} \nu^{-1/3} \enskip (<0),
\end{align*}
by condition (C3) and the estimate in \eqref{estimate:001}.

Thus
\begin{align}
&\nu(h(x) - h(\xizero)) \le \frac{h''(\xizero)}{4}\nu^{2/3} \:(<0),\\
\intertext{and since $\nu^{2/3}$ is infinitely large, we obtain}
 \sqrt{\nu} \funcexp{ \nu \, (h(x) - h(\xizero)) } &\le \nu^{3/6} \funcexp{\frac{h''(\xizero)}{4} \nu^{4/6}}  \notag \\
  &= \nu^{-1/6} \left( \nu^{4/6} \funcexp{ \frac{h''(\xizero)}{4} \nu^{4/6} } \right) \notag \\
&\le C \nu^{-1/6}, \label{estimate:002} \\
 \intertext{where $C$ is a positive constant. Hence}\\
 &\simeq 0.\notag
\end{align}

\noindent
\underline{Case 1:\enskip $a > -\infty$.}\enskip In this case, we simply have
\begin{align*}
&\int_{a}^{\xizero - \varepsilon} \sqrt{\nu}\enskip \varphi(x) \funcexp{ \nu \, (h(x) - h(\xizero))  } dx \: \simeq \: 0.
\end{align*}

\noindent
\underline{Case 2:\enskip $a = -\infty$.} \enskip For every standard $k=1,2,3,\ldots$ we have
\begin{align*}
&\int_{\xizero - k}^{\xizero - \varepsilon} \sqrt{\nu}\enskip \varphi(x) \funcexp{ \nu \, (h(x) - h(\xizero)) } dx \:\simeq \: 0.
\end{align*}
By Robinson's extension theorem, there exists some $\kappa \in \NatInf$ such that
\begin{align*}
&\int_{\xizero - \kappa}^{\xizero - \varepsilon} \sqrt{\nu}\enskip \varphi(x) \funcexp{ \nu \, (h(x) - h(\xizero))  } dx \:\simeq\: 0.
\end{align*}
By the integrability of $|\varphi(x)|$ over $[-\infty,b]$ and by \eqref{estimate:002},
\begin{align*}
& \left\lvert  \int_{-\infty}^{\xizero - \kappa} \sqrt{\nu}\enskip \varphi(x) \funcexp{\nu \, (h(x) - h(\xizero))} \, dx  \right\rvert \le C \,\nu^{-1/6}\int_{-\infty}^{\xizero - \kappa} |\varphi(x)| dx \:\simeq\: 0.
\end{align*}

Thus, whether or not $a=-\infty$, we obtain
\begin{align} 
  \int_{a}^{\xizero - \varepsilon} \sqrt{\nu} \enskip \varphi(x) \funcexp{ \nu \,(h(x) - h(\xizero)) } dx \:\simeq \: 0. \label{estimate:003} 
\end{align}

The same argument gives
\begin{align} 
  \int_{\xizero + \varepsilon}^{b} \sqrt{\nu} \enskip \varphi(x) \funcexp{ \nu \,(h(x) - h(\xizero)) } dx \:\simeq \: 0. \label{estimate:004} 
\end{align}

From \eqref{estimate:003} and \eqref{estimate:004} we obtain
\begin{align} 
  \sqrt{\nu} \int_{a}^{b} \varphi(x) \funcexp{ \nu \,(h(x) - h(\xizero))} dx \:\simeq\: \sqrt{\nu} \int_{\xizero - \varepsilon}^{\xizero + \varepsilon} \varphi(x) \funcexp{ \nu \,(h(x) - h(\xizero))} dx. \label{estimate:005} 
\end{align}

For any $x \in [\xizero - \varepsilon, \xizero + \varepsilon]$, define
\begin{align*} 
  p(x) &:= h(x) - h(\xizero), \\
  q(x) &:= \frac{h''(\xizero)}{2} (x - \xizero)^2.
\end{align*}
Clearly, $p(x)\le0$ and $q(x)\le0$.
Therefore
\begin{align} 
  \left\lvert \funcexp{\nu \, p(x)} - \funcexp{\nu \, q(x)} \right\rvert &\leq \frac{1}{\nu} \left\lvert \int_{q(x)}^{p(x)} | \exp(\nu \, t) | \, dt \right\rvert \leq \frac{|p(x) - q(x)|}{\nu}. \label{estimate:006} 
\end{align}

Since
\begin{align*} 
  p(x) &= h(x) - h(\xizero) = \int_{\xizero}^{x} \frac{d}{dt} (t - x) h'(t) \,dt = \int_{\xizero}^x (x-t) h''(t) \, dt, 
\end{align*}
and
\begin{align} 
  q(x) &= \frac{h''(\xizero)}{2} (x - \xizero)^2 = \int_{\xizero}^x (x-t) h''(\xizero) \,dt, 
\end{align}
we obtain the estimate
\begin{align*} 
  |p(x) - q(x)| &\leq \left\lvert \int_{\xizero}^x |x-t| \, |h''(t) - h''(\xizero)| \,dt \right\rvert \\
   &\leq \frac{1}{2} |h''(\xizero)| \left\lvert \int_{\xizero}^x |x-t| \,dt\right\rvert \\
    &= \frac{1}{4} |h''(\xizero)| \, |x - \xizero|^2 \\
     &\leq \frac{1}{4} |h''(\xizero)| \varepsilon^2 = \frac{1}{4} |h''(\xizero)| \nu^{-1/3}, 
\end{align*}
where the last inequality follows from \eqref{estimate:001}.

Substitution into \eqref{estimate:006} gives
\begin{align*} 
  \left\lvert \funcexp{\nu\, p(x)} - \funcexp{\nu\, q(x)} \right\rvert \le \frac{1}{4} |h''(\xizero)| \nu^{-8/6}. 
\end{align*}
Consequently,
\begin{align*} 
  \sqrt{\nu} \funcexp{\nu \,p(x)} \:\simeq\: \sqrt{\nu} \: \funcexp{\nu\,q(x)}. 
\end{align*}

By the continuity of $\varphi$ at $\xizero$, we have $\varphi(x) \simeq \varphi(\xizero)$.
Hence
\begin{align*} 
  \int_{\xizero - \varepsilon}^{\xizero + \varepsilon} \sqrt{\nu}\:\varphi(x) \funcexp{\nu \,p(x)}  dx \enskip &\simeq \int_{\xizero - \varepsilon}^{\xizero + \varepsilon} \sqrt{\nu}\: \varphi(\xizero)  \funcexp{\nu \, q(x)} dx\\
   &= \varphi(\xizero)\:\sqrt{\nu}\:\int_{\xizero - \varepsilon}^{\xizero + \varepsilon} \funcexp{\nu \, \left( \frac{h''(\xizero)}{2}(x-\xizero)^2 \right)}   dx. 
\end{align*}
Now make the change of variables
\begin{align*}
&z:= (x-\xizero)\sqrt{-\nu \frac{h''(\xizero)}{2}}.
\end{align*}
This gives
\begin{align*} 
  \int_{\xizero - \varepsilon}^{\xizero + \varepsilon} \sqrt{\nu}\:\varphi(x) \:\funcexp{\nu \, p(x)} \, dx &= \varphi(\xizero) \:\sqrt{\nu}\: \sqrt{-\frac{2}{\nu \:h''(\xizero)}} \int_{|z| \le R} e^{-z^2} \,dz\\
   &= \varphi(\xizero) \sqrt{-\frac{2}{h''(\xizero)}} \int_{|z| \le R} e^{-z^2} \,dz, \tag{*} \label{ex:007} 
\end{align*} 
where
\begin{align*}
  &R := \varepsilon \sqrt{\nu \frac{-h''(\xizero)}{2}} = \nu^{1/3} \sqrt{-\frac{h''(\xizero)}{2}}.
\end{align*}
Since $R$ is infinitely large, the integral in \eqref{ex:007} is infinitely close to $\displaystyle \int_{-\infty}^{\infty} e^{-z^2}\,dz = \sqrt{\pi}$.
Therefore
\begin{align*} 
  \eqref{ex:007} \simeq \varphi(\xizero) \sqrt{-\frac{2}{h''(\xizero)}} \sqrt{\pi} = \varphi(\xizero) \sqrt{-\frac{2\pi}{h''(\xizero)}}. 
\end{align*} 
Combining this approximation with \eqref{estimate:005}, we conclude that \begin{align*} \sqrt{\nu} \int_a^b \varphi(x) \funcexp{ \nu \, (h(x) - h(\xizero)) } dx \:\simeq\: \varphi(\xizero) \sqrt{-\frac{2\pi}{h''(\xizero)}}. 
\end{align*}

Since $\nu$ was arbitrary, \eqref{ex:purpose} holds for any standard $a, b, \varphi(x), h(x)$ and $\xizero$.
By the transfer principle, \eqref{ex:purpose} extends to all (possibly nonstandard) $a, b, \varphi(x), h(x)$ and $\xizero$.
This completes the proof. 
\end{proof}

\section{A Generalization}\label{sec:generalization}
The proof of Theorem \ref{th:Laplace} extends to the following more general formula.

\begin{theorem}[Generalized Laplace's theorem]\label{th:gen-Laplace}
Let $[a,b]\subseteq \Real$ be a finite or infinite interval, and let $\varphi(x)$ and $h(x)$ be functions defined on this interval. 
Suppose that $\xizero \in (a,b)$ and that $m$ is a natural number with $m \ge 1$.

Assume that these data satisfy (C1), (C3), (C4), and (C5) of Theorem \ref{th:Laplace}, together with the following modified smoothness condition:
\begin{enumerate}
  \item [(C2')] The function $h(x)$ satisfies $h^{(k)}(\xizero) = 0$ 
  for $1 \le k < 2m$, and $h^{(2m)}(\xizero) <0$.
\end{enumerate}
Then, as $n \to \infty$, the following asymptotic formula holds:
\begin{align*}
  &\int_a^b \varphi(x) \funcexp{n\, h(x)} dx \:\sim\: \varphi(\xizero) \funcexp{n\, h(\xizero)} \frac{\Gamma \left(\frac{1}{2m} \right)}{m} \left( -\frac{(2m)!}{n\, h^{(2m)}(\xizero) } \right)^\frac{1}{2m}.
\end{align*}
\end{theorem}

The proof uses Taylor's theorem with integral remainder.
\begin{lemma} \label{lemma:taylor-expansion-with-integral-remainder} 
  Let $g$ be sufficiently differentiable on an interval containing $\xizero$.
Then, for every $x$ in the interval,
\begin{align} 
  g(x) - g(\xizero) = \sum_{k = 1}^{N-1} \frac{g^{(k)}(\xizero)}{k!} (x - \xizero)^k + \frac{1}{N!} \int_{\xizero}^x g^{(N)}(t)(x - t)^{N-1} \, dt. \label{2025-02-27-010} 
\end{align} 
\end{lemma}
\begin{proof} 
Fix a positive integer $k$.
Integration by parts gives
\begin{align} 
  \int_{\xizero}^x g^{(k)}(t)\,(x-t)^{k-1}\,dt &= - \frac{1}{k} \int_{\xizero}^x g^{(k)}(t)\: \frac{d}{dt} \left( (x-t)^{k} \right) \, dt \notag\\
   &= - \frac{1}{k} \left[ g^{(k)}(t) \: (x-t)^{k} \right]_{t = \xizero}^{t = x} + \frac{1}{k} \int_{\xizero}^x g^{(k+1)}(t)\:(x-t)^{k}\,dt \notag\\
    &= \frac{g^{(k)}(\xizero)}{k} (x - \xizero)^k + \frac{1}{k} \int_{\xizero}^x g^{(k+1)}(t)\:(x-t)^{k}\,dt. 
\end{align}
Starting from
\begin{align*} 
  g(x) - g(\xizero) =  \int_{\xizero}^x g^{(1)}(t)\: (x-t)^{0} \,dt, 
\end{align*} 
and applying the preceding identity repeatedly, we obtain the formula. 
\end{proof}

\begin{proof}[Proof of Theorem \ref{th:gen-Laplace}]

Since the argument is parallel to the proof of Theorem \ref{th:Laplace}, some details are omitted.

We first assume that $a$, $b$, $\varphi$, $h$, and $\xizero$ are standard.
Then the desired relation is equivalent to
\begin{align} 
  \forall \nu \in \NatInf\, \quad \nu^\theFraction \int_a^b \varphi(x) \funcexp{ \nu \, (h(x) - h(\xizero)) } dx \simeq \varphi(\xizero)\frac{\Gamma \left( \theFraction \right)}{m} \left( -\frac{(2m)!}{h^{(2m)}(\xizero)}\right)^\theFraction. \label{ex:purpose'} 
\end{align}

Fix an arbitrary $\nu\in \NatInf$ and set $\varepsilon := \nu^{-1/6m^2}$.
We shall use the following estimate.
\begin{lemma}%
\label{lemma:esimate:001'}
For every $x \in [\xizero - \varepsilon, \xizero + \varepsilon]$, one has
\begin{align} 
\frac{3}{2} h^{(2m)}(\xizero) \leq  h^{(2m)}(x) \leq \frac{1}{2} h^{(2m)}(\xizero)\,(<0). \label{estimate:001'}
\end{align}
\end{lemma}
\begin{proof}
The proof is the same as that of Lemma \ref{lemma:esimate:001}.
\end{proof}


As before, we split the integral into three parts:
\begin{align}
\nu^\theFraction \int_a^b \varphi(x) \funcexp{ \nu \, (h(x) - h(\xizero)) } dx &= \nu^\theFraction \int_a^{\xizero - \varepsilon} \varphi(x) \funcexp{ \nu \, (h(x) - h(\xizero)) } dx \notag\\
  &+  \nu^\theFraction \int_{\xizero - \varepsilon}^{\xizero + \varepsilon} \varphi(x) \funcexp{ \nu \, (h(x) - h(\xizero)) } dx \notag\\
  &+  \nu^\theFraction \int_{\xizero + \varepsilon}^b \varphi(x) \funcexp{ \nu \, (h(x) - h(\xizero)) } dx. \label{eqn:026'}
\end{align}

We estimate the first term on the right-hand side of \eqref{eqn:026'}.
For any $x \in [a, \xizero - \varepsilon]$, we have
\begin{align*}
h(x)  - h(\xizero) &\le h(\xizero - \varepsilon) - h(\xizero) \notag\\
&= \sum_{k = 1}^{2m-1} \underbrace{h^{(k)}(\xizero)}_{=0} \frac{(-\varepsilon)^k}{k!} \enskip + \enskip \frac{1}{(2m)!}\int_{\xizero}^{\xizero - \varepsilon} h^{(2m)}(t)\: (\xizero - \varepsilon - t)^{2m-1}\, dt\\
&= \frac{1}{(2m)!} \int_{\xizero - \varepsilon}^{\xizero} h^{(2m)}(t) \: \left\lvert t - (\xizero - \varepsilon)\right\rvert^{2m-1}\, dt\\
&\le \frac{1}{2}\frac{1}{(2m)!} h^{(2m)}(\xizero) \: \varepsilon^{2m-1} \int_{\xizero - \varepsilon}^{\xizero} dt\\
&= \frac{1}{2(2m)!}h^{(2m)}(\xizero) \:\varepsilon^{2m} = \frac{h^{(2m)}(\xizero)}{2(2m)!}\: \nu^{-1/3m}
\end{align*}
by Taylor's theorem, condition (C3) and \eqref{estimate:001'}.
Thus, $\displaystyle\nu(h(x) - h(\xizero)) \le \frac{h^{(2m)}(\xizero)}{2(2m)!}\nu^{1-\frac{1}{3m}} \:(<0)$ and, since $\nu^{1-\frac{1}{3m}}$ is infinitely large,
\begin{align}
\nu^\theFraction\enskip  \funcexp{ \nu \, (h(x) - h(\xizero)) } &\le \nu^\frac{2}{6m}  \funcexp{ \frac{h^{(2m)}(\xizero)}{2(2m)!}\nu^{1-\frac{2}{6m}}}\notag\\
& = \nu^{\frac{5}{6m}-1} \left( \nu^{1-\frac{2}{6m}} \funcexp{ \frac{h^{(2m)}(\xizero)}{2(2m)!} \nu^{1-\frac{2}{6m}} }  \right) \notag\\
& \le  C \,\nu^{\frac{5}{6m}-1}, \label{estimate:002'}\\
\intertext{for some positive constant $C$. Hence}
&\simeq 0.\notag
\end{align}

By the same argument as in the proof of Theorem \ref{th:Laplace}, we obtain
\begin{align}
&\int_{a}^{\xizero - \varepsilon} \nu^\theFraction \enskip \varphi(x) \funcexp{ \nu \, (h(x) - h(\xizero)) } dx \:\simeq \: 0\label{estimate:003'}\\
\intertext{and}
&\int_{\xizero + \varepsilon}^{b} \nu^\theFraction \enskip \varphi(x) \funcexp{ \nu \,(h(x) - h(\xizero)) } dx \:\simeq \: 0.\label{estimate:004'}
\end{align}

From \eqref{estimate:003'} and \eqref{estimate:004'} we have
\begin{align}
& \nu^\theFraction \int_{a}^{b} \varphi(x) \funcexp{ \nu \,(h(x) - h(\xizero))} dx \:\simeq\: \nu^\theFraction \int_{\xizero - \varepsilon}^{\xizero + \varepsilon} \varphi(x) \funcexp{ \nu \,(h(x) - h(\xizero))} dx.\label{estimate:005'}
\end{align}

For $x \in [\xizero - \varepsilon, \xizero + \varepsilon]$, set
\begin{align*}
&p(x):= h(x) - h(\xizero), \enskip q(x) := \frac{h^{(2m)}(\xizero)}{(2m)!}(x-\xizero)^{2m}.
\end{align*}
Clearly, $p(x),q(x)\le0$.
As before,
\begin{align}
&\left\lvert \funcexp{\nu\, p(x)} - \funcexp{\nu\, q(x)} \right\rvert \le \frac{1}{\nu} \left\lvert\int_{q(x)}^{p(x)} | \exp(\nu \  t) |\, dt \right\rvert \le \frac{|p(x) - q(x)|}{\nu}.\label{estimate:006'}
\end{align}

It is immediate from Taylor's formula that
\begin{align*}
p(x)&= \frac{1}{(2m-1)!} \int_{\xizero}^x (x-t)^{2m-1} h^{(2m)}(t) dt
\end{align*}
and
\begin{align}
q(x)&= \frac{1}{(2m-1)!}\int_{\xizero}^x (x-t)^{2m-1} h^{(2m)}(\xizero) \,dt.
\end{align}
Thus we have
\begin{align*}
|p(x) - q(x)| &\le \frac{1}{(2m-1)!}\left\lvert \int_{\xizero}^x |x-t|^{2m-1} \, \left\lvert h^{(2m)}(t) - h^{(2m)}(\xizero)\right\rvert \,dt  \right\rvert\\
&\le \frac{1}{2(2m-1)!} \left\lvert h^{(2m)}(\xizero) \right\rvert \left\lvert \int_{\xizero}^x |x-t|^{2m-1} \,dt\right\rvert\\
 &= \frac{1}{2(2m)!} \left\lvert h^{(2m)}(\xizero) \right\rvert \, |x - \xizero|^{2m} \le \frac{1}{2(2m)!}\left\lvert h^{(2m)}(\xizero) \right\rvert \varepsilon^{2m} = \frac{1}{2(2m)!}\left\lvert h^{(2m)}(\xizero) \right\rvert\nu^{-\frac{1}{3m}},
\end{align*}
by the estimate \eqref{estimate:001'}.
Substitution into \eqref{estimate:006'} gives
\begin{align*}
&\left\lvert \funcexp{\nu\, p(x)} - \funcexp{\nu\, q(x)}\right\rvert \le \frac{1}{2(2m)!}\left\lvert h^{(2m)}(\xizero) \right\rvert \nu^{-1-\frac{2}{6m}}.
\end{align*}
Hence
\begin{align*}
& \nu^\theFraction \:\funcexp{\nu\,p(x)} \:\simeq\: \nu^\theFraction \: \funcexp{\nu\,q(x)}.
\end{align*}
By the continuity of $\varphi$ at $\xizero$, we have $\varphi(x) \simeq \varphi(\xizero)$. Hence
\begin{align*}
\int_{\xizero - \varepsilon}^{\xizero + \varepsilon} \nu^\theFraction \:\varphi(x) \:\funcexp{\nu \, p(x)} \, dx \enskip &\simeq \int_{\xizero - \varepsilon}^{\xizero + \varepsilon} \nu^\theFraction\: \varphi(\xizero) \: \funcexp{\nu \, q(x)} \, dx\\
 &= \varphi(\xizero)\:\nu^\theFraction\:\int_{\xizero - \varepsilon}^{\xizero + \varepsilon}  \funcexp{\nu \, \left( \frac{h^{(2m)}(\xizero)}{(2m)!)}(x-\xizero)^{2m} \right)}  \, dx\\
\intertext{Put $\displaystyle z:= (x-\xizero)\left(-\nu \frac{h^{(2m)}(\xizero)}{(2m)!} \right)^\theFraction$.}
&= \varphi(\xizero) \:\nu^\theFraction\: \left( \nu \frac{-h^{(2m)}(\xizero)}{(2m)!}\right)^{-\theFraction} \int_{|z| \le R'} e^{-z^{2m}} \,dz\\
&= \varphi(\xizero) \: \left( \frac{-h^{(2m)}(\xizero)}{(2m)!}\right)^{-\theFraction} \int_{|z| \le R'} e^{-z^{2m}} \,dz, \tag{**}\label{ex:007'}
\end{align*}
where
\begin{align*}
&R' := \varepsilon \left(\nu \frac{-h^{(2m)}(\xizero)}{(2m)!} \right)^\theFraction = \nu^{-\frac{1}{6m^2} + \frac{3m}{6m^2}} \left(-\frac{h^{(2m)}(\xizero)}{(2m)!)} \right)^\theFraction.
\end{align*}
Since $R'$ is infinitely large, the integral in \eqref{ex:007'} is infinitely close to $\displaystyle \int_{-\infty}^{\infty} e^{-z^{2m}}\,dz = \frac{1}{m}\Gamma \left(\theFraction \right)$. Therefore,
\begin{align*}
&\eqref{ex:007'} \simeq \varphi(\xizero) \frac{\Gamma \left( \theFraction \right)}{m} \left( -\frac{(2m)!}{h^{(2m)}(\xizero)} \right)^\theFraction.
\end{align*}
Combining \eqref{estimate:005'} with the preceding estimate, we obtain
\begin{align*}
&\sqrt{\nu} \int_a^b \varphi(x) \funcexp{ \nu \,(h(x) - h(\xizero))  }\,dx \:\simeq\: \varphi(\xizero) \frac{\Gamma \left(\theFraction \right)}{m} \left( -\frac{(2m)!}{h^{(2m)}(\xizero)} \right)^\theFraction.
\end{align*}

Since $\nu$ was arbitrary, \eqref{ex:purpose'} is shown for any standard $a, b, \varphi(x), h(x)$ and $\xizero$.
By the transfer principle, \eqref{ex:purpose'} holds for any (possibly nonstandard) $a, b, \varphi(x), h(x)$ and $\xizero$.
This completes the proof.
\end{proof}

\section{Note}

In our proof of Theorem \ref{th:Laplace}, we could take $\varepsilon = \nu^{-1/(4+\delta)}$, by any standard positive $\delta$.

The statement of the theorem is adapted from Problem 201 in Part Two of \cite{Polya-and-Szego-1}.
More precisely, conditions (C3) and (C4) have been added to the hypotheses, while one condition from the original statement has been omitted.
Conditions (C3) and (C4) appear to hold in typical applications of the formula.

Laplace's theorem has already been proved by nonstandard methods in \cite{vdBerg-Asymp} from a broader point of view.
That proof uses general techniques involving external numbers.

After the initial submission, the author became aware that D. S. Jones also presents a proof of Laplace's theorem using nonstandard analysis in his book \cite{Jones-Asymp}, and a reference has been added accordingly.

\section*{Change History}

\begin{itemize}
  \item [v1.] First version.
  \item [v2.] \begin{itemize}
    \item The introductory heuristic discussion has been rewritten to improve readability.
    \item Conditions (C3) and (C4) have been simplified.
    \item An explanation of \eqref{estimate:001}, which had previously been omitted as being straightforward, has been added.
    \item The cumbersome estimates in the proof of Theorem \ref{th:gen-Laplace} have been substantially simplified by reducing them to Taylor's theorem with integral remainder.
    \item Minor adjustments have been made to several formulas and notational conventions.
    \item A note has been added concerning the choice of the infinitesimal width.
    \item A reference to the book by D. S. Jones has been added.
    \item  The English has been polished with the help of an LLM.
  \end{itemize}
  
\end{itemize}

\bibliographystyle{alpha}
\bibliography{bib20260608}

@article{Nelson1977,
	author = {Edward Nelson},
	title = {Internal Set Theory : A New Approach to Nonstandard Analysis},
	journal = {Bull. Amer. Math. Soc.},
	year = 1977,
	volume = {83},
	number = {6},
	pages = {1165--1198},
}

@book{vdBerg-Asymp,
	author = {Imme van den Berg},
	title = {Nonstandard Asymptotic Analysis},
	publisher = {Springer},
	year = {1980}
}

@book{Jones-Asymp,
	author = {D.S.Jones},
	title = {Introduction to Asymptotics},
	subtitle = {A Treatment Using Nonstandard Analysis},
	publisher = {World Scientific},
	year = {1997}
}

@book{Polya-and-Szego-1,
	author = {George P\'{o}lya and Gabor Szeg\"{o}},
	title = {Problems and Theorems in Analysis, Vol.1},
	publisher = {Springer},
	year = 1970
}

\end{document}